\documentclass[11pt,reqno]{amsart}

\usepackage[utf8]{inputenc}
\usepackage[margin=1.25in]{geometry}
\usepackage{hyperref}
\usepackage{appendix}
\usepackage{amsfonts}
\usepackage{amsthm}
\usepackage{amssymb}
\usepackage{stmaryrd} 
\usepackage{amsmath}
\usepackage{amsthm}
\usepackage{mathrsfs}
\usepackage{lipsum} 

\theoremstyle{plain}
\newtheorem{theorem}{Theorem}[section]

\newtheorem{lemma}[theorem]{Lemma}

\theoremstyle{remark}

\newtheorem{remark}[theorem]{Remark}

\usepackage{biblatex} 
\numberwithin{equation}{section}
\title[Deviation top eigenvalue tridiagonal matrices]{Deviation of top eigenvalue for some tridiagonal matrices under various moment assumptions}

\author{Yi HAN}
\address{Department of Pure Mathematics and Mathematical Statistics, University of Cambridge.
}
\email{yh482@cam.ac.uk}
\thanks{Supported by EPSRC grant EP/W524141/1.}

\begin{document}

\begin{abstract}
 Symmetric tridiagonal matrices appear ubiquitously in mathematical physics, serving as the matrix representation of discrete random Schrödinger operators. In this work we investigate the top eigenvalue of these matrices in the large deviation regime, assuming the random potentials are on the diagonal with a certain decaying factor $N^{-\alpha}$, and the probability law $\mu$ of the potentials satisfy specific decay assumptions. We investigate two different models, one of which has random matrix behavior at the spectral edge but the other does not. Both the light-tailed regime, i.e. when $\mu$ has all moments, and the heavy-tailed regime are covered. Precise right tail estimates and a crude left tail estimate are derived. In particular we show that when the tail $\mu$ has a certain decay rate, then the top eigenvalue is distributed as the Fréchet law composed with some deterministic functions. The proof relies on computing one point perturbations of fixed tridiagonal matrices.
\end{abstract}

\maketitle

\section{Introduction}
 The theory of random Schrödinger operators gained much success in describing various phenomena in mathematical physics. A 1-d random Schrödinger operator in its finite volume version can be described as the following operator defined on $\{0,1,\cdots,N\}$
\begin{equation}
    H_N\varphi(k)=\varphi(k-1)+\varphi(k+1)+\mathfrak{a}(k)\varphi(k),\quad \varphi(0)=\varphi(N)=0,
\end{equation}
where $\mathfrak{a}(k),k=1,\cdots,N$, are certain random potentials. It corresponds to the matrix representation 
\begin{equation}
    H_N=H_N^\infty+\Lambda_N^0,
\end{equation}

where 
\begin{equation}\label{1.41.4}
H_N^\infty=\begin{pmatrix}
0&1&0&\cdots&0\\
1&0&1&\ddots&0\\
\vdots&\ddots&\ddots&\ddots&\vdots\\
0&\cdots&1&0&1\\
0&\cdots&0&1&0
    \end{pmatrix},\end{equation}
and \begin{equation}
    \Lambda_N^0=\operatorname{diag}\left(\mathfrak{a}(1),\cdots,\mathfrak{a}(N)\right).
\end{equation}

Under some slight assumptions on the potentials $\mathfrak{a}(\cdot)$, \cite{minami1996local} proves that the bulk eigenvalues of $H_N$ converge to a Poisson process after an appropriate re-scaling. Further, assuming that $\mathfrak{a}$ has a doubly exponential type upper tail distribution, it is known \cite{biskup2016eigenvalue} that $\lambda_1(N)$, the top eigenvalue of $H_N$, converges to a Gumbel distribution of max-order class.

To investigate the transition from localized phase to de-localized phase, a variant of $H_N$ was introduced in \cite{kiselev1998modified} (for a modified model) and \cite{WOS:000308042700006} where the potentials are vanishing with $N$, i.e.
\begin{equation}\label{decayingdigaonal}
H_N^\alpha=H_N^\infty+\Lambda_N^\alpha,\quad     \Lambda_N^\alpha=\operatorname{diag}(N^{-\alpha}\mathfrak{a}(1),\cdots, N^{-\alpha}\mathfrak{a}(N)).
\end{equation}
The transition from localized to de-localized phase has been verified the bulk statistics at the critical value $\alpha=\frac{1}{2}$, see for example \cite{nakano2017fluctuation}.

In this work we investigate the asymptotics of the top eigenvalue of $H_N^\alpha$ with decaying diagonal potentials \eqref{decayingdigaonal}. When the potentials $\mathfrak{a}$ are not too heavy-tailed, it is easy to check that the top eigenvalue of $H_N^\alpha$ converges to 2 almost surely, and when $\mathfrak{a}$ has heavier tails, the top eigenvalue can be arbitrarily large with positive probability. We identify the large deviation profile of the top eigenvalue in the former light-tailed case, then identify the threshold of criticality, and finally derive the law of the top eigenvector in the critical case. When $\mathfrak{a}(\cdot)$ has even heavier tails, we verify that the top eigenvalue has a Poisson distribution.

This work is partially inspired by the recent progress in large deviation of Wigner matrices. When the matrix has sub-Gaussian entries with a bound on its Laplace transform, the large deviation profile of the largest eigenvalue was derived via spherical integral in \cite{WOS:000542157900013}, where it was the \textit{collective} behavior of random potentials, instead of the behavior of one outlier potential, that accounts the large deviation profile. When the random variables have a tail heavier than Gaussian, the large deviation of spectral statistics was derived in \cite{articledeviationcapto} (and see \cite{10.1214/16-EJP4146} for deviation of the largest eigenvalue in this regime) where the deviation is accounted for by few outlier potentials taking very large values. For heavy-tailed Wigner matrices, the distribution of the top eigenvalue was also well studied by many authors. When matrix entries have finite fourth moment, the top eigenvalue sticks to the edge of semicircle law \cite{bai1988necessary}, and the rescaled top eigenvalue has a Tracy-Widom fluctuation at the edge\cite{Lee2012ANA}. When the matrix entries do not have a finite fourth moment, it is verified in \cite{Soshnikov2004} and \cite{article} that the largest eigenvalue has a Poisson distribution. Finally, when the density of the matrix entries have a tail decay $x^{-4}dx$, the distribution of the top eigenvalue is explicitly determined in \cite{diaconu2023more}. For the random Schrödinger operator $H_N^\alpha$, we will show that the large deviation profile, and the asymptotic distribution of the top eigenvalue is always governed by outlier potentials of $\mathfrak{a}$ taking unusually large values.

We also study a model of close relevance: consider the matrix 
\begin{equation}\label{1.2001}
    G_N^\infty= 
    \begin{pmatrix}
   0&\sqrt{(N-1)/N}&0&&\\
   \sqrt{(N-1)/N} & 0 &  \sqrt{(N-2)/N}\\
    &\ddots&\ddots&&\\
        &&0&\sqrt{2/N}&0\\
    &&\sqrt{2/N}&0&\sqrt{1/N}\\
    &&0&\sqrt{1/N}&0
    \end{pmatrix}.
\end{equation} 
This model is closely related to the matrix model of beta ensembles, introduced by Dumitriu and Edelman in \cite{dumitriu2002matrix}. Indeed, $G_N^\infty$ can be thought of the zero temperature model in this beta-parameterized matrix family \eqref{fucky}. We are also interested in random operators of the following form
\begin{equation}
   G_N^\infty+\Lambda_N^\alpha,
\quad \Lambda_N^\alpha=\operatorname{diag}(N^{-\alpha}\mathfrak{a}(1),\cdots, N^{-\alpha}\mathfrak{a}(N)),\end{equation}
but in order to foster better analytic treatment, we will instead work with an orthogonal invariant version
\begin{equation}\label{matrixformofG}
G_N^\alpha\overset{\text{def}}=G_N^\infty+U_N\Lambda_N^\alpha U_N^T,\end{equation}
where $U_N$ is a Haar distributed orthogonal matrix of size $N$, independent of the random variables $\mathfrak{a}(\cdot)$, and $U_N^T$ denotes the transpose of $U_N$. We will verify that the large deviation profiles of the top eigenvalue of $H_N^\infty$ and $G_N^\infty$ are very similar, except that the precise expression of large deviation rate function is different.

\subsection{Statement of main results}

The main results of this paper are outlined as follows:

\begin{theorem}\label{1sttheorem}
    Consider the random operator $H_N^\alpha$ defined in \eqref{decayingdigaonal}, and denote by $\lambda_1(N)$ the largest eigenvalue of $H_N^\alpha$. Assume that the random potentials $\mathfrak{a}(1),\cdots,\mathfrak{a}(N)$ are i.i.d. with distribution $\mathfrak{a}$. Then 
    \begin{enumerate}
        \item  (Weibull distribution) Assume that for some $\beta>0$, some $C>0$, $0<C_1<C_2<\infty$ we have for any $t>1$,
        \begin{equation}\label{eq1,81}
         C_1\exp(-Ct^\beta)\leq \mathbb{P}(\mathfrak{a}>t)\leq   C_2\exp(-Ct^\beta),
        \end{equation}
         \begin{equation}\label{eq1.91}
         \mathbb{P}(\mathfrak{a}<-t)\leq   C_2\exp(-Ct^\beta),
        \end{equation}
        
        then for any $\lambda>2$,  we have the \textbf{upper tail} estimate \begin{equation}
            \lim_{N\to\infty}\frac{\log(\mathbb{P}(\lambda_1(N)>\lambda))}{N^{\alpha\beta}}=-C\left(\sqrt{\lambda^2-4}\right)^\beta,
        \end{equation}
and for any $\lambda<2$, we have the \textbf{lower tail} estimate: there exists some constants $C_3,C_4>0$, and when $\beta<2$ there exists some sufficiently small $c>0$ such that
            \begin{equation}\label{lefttailguud}
   \mathbb{P}(\lambda_1(N)<\lambda)\leq\begin{cases}
C_3e^{-C_4N^{\alpha\beta+1}},\quad \beta\in [2,\infty),\\
C_3e^{-C_4N^{\alpha\beta+c}},\quad \beta\in(0,2).
   \end{cases}
        \end{equation}
 The constant $C_4>0$ depends on $\lambda$, and can be written as an explicit function of $\lambda$. See Remark \ref{explicitlefttail} for a discussion.
        
        \item (Sub-critical)
        Assume that for some $\beta>\frac{1}{\alpha}$, $C>0,C_1>0$ we have for all $t>1$,
        \begin{equation}\label{eq1.121}
            \mathbb{P}(\mathfrak{a}>t)=  Ct^{-\beta},
        \end{equation} 
         \begin{equation}\label{eq1.131}
            \mathbb{P}(\mathfrak{a}<-t)\leq   C_1t^{-\beta},
        \end{equation}
        then for any $\lambda>2$, we have the \textbf{upper tail} estimate:
        \begin{equation}
    \lim_{N\to\infty} N^{\alpha\beta-1}\mathbb{P}(\lambda_1(N)>\lambda)=\left(\sqrt{\lambda^2-4}\right)^{-\beta},
        \end{equation}
        and for any $\lambda<2$ we have the \textbf{lower tail} estimate: when $\beta>2$, for any $\lambda<2$ and any $\epsilon>0$ we can find $C_2=C_2(\epsilon)>0$ depending on $\epsilon$ and $\lambda$ such that
         \begin{equation}\label{1.151}
   \mathbb{P}(\lambda_1(N)<\lambda)\leq (C_2N)^{-(\alpha(\beta-\epsilon)+1)}
        \end{equation}
        and when $\beta\in(0,2]$, for any $\lambda<2$ we can find sufficiently small $c>0$ and some $C_3=C_3(\beta,c,\lambda)$ such that 
        \begin{equation}\label{1.161}
   \mathbb{P}(\lambda_1(N)<\lambda)\leq (C_3N)^{-c}.
        \end{equation}
        \item (Critical regime) Assume there exists some $C>0,C_1>0$ such that for any $t>1$,
        \begin{equation}\label{1.171}
            \mathbb{P}(\mathfrak{a}>t)=Ct^{-\frac{1}{\alpha}}, 
        \end{equation} 
         \begin{equation}\label{1.181}
            \mathbb{P}(\mathfrak{a}<-t)\leq   C_1t^{-\frac{1}{\alpha}},
        \end{equation}
        then almost surely
        \begin{equation}
            \lambda_1(N)\overset{law}\to \sqrt{\xi^2+4}, \quad N\to\infty,
        \end{equation}
where $\xi$ is a non-negative-valued random variable that satisfies, for any $\lambda>0$,
\begin{equation}
    \mathbb{P}(\xi>\lambda)=1-e^{-C\lambda^{-\frac{1}{\alpha}}}.
\end{equation}

        \item (Randomness dominating)  
         Assume that for some $0<\beta<\frac{1}{\alpha}$, $C>0,C_1>0$ we have for all $t>1$,
          \begin{equation}\label{1.211}
            \mathbb{P}(\mathfrak{a}>t)=   Ct^{-\beta},
        \end{equation}  
         \begin{equation}\label{1.221}
            \mathbb{P}(\mathfrak{a}<-t)\leq   C_1t^{-\beta},\end{equation}
        then 
        \begin{equation}\label{1.231}
         N^{\alpha-\frac{1}{\beta}} \lambda_1(N)\overset{law}\to \xi, \quad N\to\infty,
        \end{equation} where $\xi$ is a non-negative valued random variable satisfying for any $\lambda>0$,
        \begin{equation}
    \mathbb{P}(\xi>\lambda)=1-e^{-C\lambda^{-\beta}}.
\end{equation}
    \end{enumerate} 
\end{theorem}

\begin{remark}
    The upper tail estimates in cases (1) and (2) are sharp, yet the lower tail estimates derived in cases (1) and (2) are crude and may not be optimal. A related example is the lower tail of KPZ equation with narrow wedge initial condition \cite{tsai2022exact}, where the exact lower tail estimate was derived via computing the exponential moment of the Airy-2 process, which admits an interpretation via the stochastic Airy operator. The stochastic Airy operator is closely related to the theme of this paper, yet it is defined via Brownian motion and thus a lot of analytical techniques (like the Girsanov transform) can be used. In our paper, we consider different moment assumptions on the random variables, so that the analytical techniques in the context of stochastic Airy operator are not available to us for deriving a sharper lower tail estimate.
\end{remark}
There has also been some recent investigations on the matrix model \eqref{decayingdigaonal}, where a scaling limit of top eigenvalue (when $\alpha=\frac{3}{2}$) was derived in \cite{han2023universal}.

    We can also derive the limiting distribution of the second, third, fourth, etc. largest eigenvalues of $H_N^\alpha$, denoted by $\lambda_i(N),i\geq 2$. All the derivations depend on the distribution of the largest value of $\mathfrak{a}(1),\cdots,\mathfrak{a}(N)$ and the finite rank perturbation formula in Lemma \ref{lemma2.1}, so that the distributions of the second, third, fourth, etc. eigenvalues can be written out explicitly. 
    
    In cases (1) and (2), we can check that the deviation probabilities of  $\lambda_i(N),i\geq 2$ are negligible compared to $e^{-N^{\alpha\beta}}$ and $N^{1-\alpha\beta}$ respectively, and is thus negligible compared to deviation of $\lambda_1(N)$. In cases (3) and (4) however, deviations of $\lambda_i(N),i\geq 2$ have the same magnitude as that of $\lambda_1(N)$, and we have the following point process characterization:

    \begin{theorem}\label{theorem1.1ladd} In case (3) of Theorem \ref{1sttheorem}, for any $\epsilon>0$, the largest eigenvalues of $H_N^\alpha$ in $(2+\epsilon,\infty)$ converge to a random point process which is the image under the map $x\mapsto\sqrt{x^2+4}$ of the Poisson process on $[0,\infty)$ with intensity  $
     \frac{C}{\alpha x^{\frac{1}{\alpha}+1}},\quad x>0.$

    In case (4) of Theorem \ref{2ndtheorem}, the point process $$N^{\alpha-\frac{1}{\beta}}\lambda_1(N)\geq N^{\alpha-\frac{1}{\beta}}\lambda_2(N)\geq\cdots, $$ where $\lambda_1(N)\geq\lambda_2(N)\geq \cdots$ denote the eigenvalue of $H_N^\alpha$ in decreasing order, converges to the Poisson point process on $[0,\infty)$ with intensity 
    $
     \frac{C\beta}{ x^{\beta+1}},\quad x>0.
    $
        
    \end{theorem}

\begin{remark} The Poisson distribution claimed in (4) has been verified for several other random matrix models with heavy tail distribution, see \cite{Soshnikov2004}. Meanwhile, in all these claims, the assumption that the tail estimate for $\mathfrak{a}$ holds for all $t>1$ can be easily changed to holding for all $t$ larger than a fixed number.
\end{remark}

In the following we summarize the results on top eigenvalues concerning the random operator \eqref{matrixformofG}. We first introduce a convenient notation: for any $\lambda\geq 2$, denote by $f(\lambda)$ the solution to $$f(\lambda)+\frac{1}{f(\lambda)}=\lambda$$ that satisfies $f(\lambda)\geq 1$.

\begin{theorem}\label{2ndtheorem}
    Consider the random operator $G_N^\alpha$ defined in \eqref{matrixformofG}, where $U_N$ is a Haar distributed unitary matrix independent of the random variables $\mathfrak{a}(\cdot)$. We denote by $\lambda_1(N)$ the largest eigenvalue of $G_N^\alpha$. Assume the random potentials $\mathfrak{a}(1),\cdots,\mathfrak{a}(N)$ are i.i.d. with distribution a real-valued random variable $\mathfrak{a}$. Then 
    \begin{enumerate}
        \item  (Weibull distribution) Assume that the random variable $\mathfrak{a}(\cdot)$ satisfies Assumptions \eqref{eq1,81}, \eqref{eq1.91}. Then for any $\lambda>2$,  we have the upper tail estimate \begin{equation}\label{limitsss}
            \lim_{N\to\infty}\frac{\log(\mathbb{P}(\lambda_1(N)>\lambda))}{N^{\alpha\beta}}=-C\left(f(\lambda)\right)^\beta,
        \end{equation}
and for any $\lambda<2$ we have the lower tail estimate as in \eqref{lefttailguud}.
        
        \item (Sub-critical)
        Assume that the random variables satisfy \eqref{eq1.121} and \eqref{eq1.131}. Then
        for any $\lambda>2$, we have the upper tail estimate:
        \begin{equation}\label{agagagagagagrqag}
    \lim_{N\to\infty} N^{\alpha\beta-1}\mathbb{P}(\lambda_1(N)>\lambda)=\left(f(\lambda)\right)^{-\beta},
        \end{equation}
        and for any $\lambda<2$ we have the lower tail estimates in exactly the same form as in \eqref{1.151}, \eqref{1.161} (though the constants $C_2$, $C_3$ there may be slightly different).

        \item (Critical regime) Assume that the random variables satisfy \eqref{1.171}, \eqref{1.181} Then 
        \begin{equation}
            \lambda_1(N)\overset{law}=\begin{cases} 2,&\xi\leq 1,\\\xi+\frac{1}{\xi},&\xi>2,
            \end{cases}
        \end{equation}
where $\xi$ is a non-negative-valued random variable that satisfies, for any $\lambda>0$,
\begin{equation}
    \mathbb{P}(\xi>\lambda)=1-e^{-C\lambda^{-\frac{1}{\alpha}}}.
\end{equation}
More generally, for any $d\in\mathbb{N}_+$, the largest $d$ eigenvalues of $G_N^\alpha$ in $(2,
\infty)$ converge to a random point process which is the image under the map $$x\mapsto \begin{cases}2&x<1\\x+\frac{1}{x}&x\geq 1\end{cases}$$ of the largest $d$ points in the Poisson process on $[0,\infty)$ with intensity  $
     \frac{C}{\alpha x^{\frac{1}{\alpha}+1}},\quad x>0.$

        \item (Randomness dominating)  
         Assume that the random variables satisfy \eqref{1.211}, \eqref{1.221}. Then the top eigenvalue $\lambda_1(N)$ satisfies \eqref{1.231}. Meanwhile, the point process $$N^{\alpha-\frac{1}{\beta}}\lambda_1(N)\geq N^{\alpha-\frac{1}{\beta}}\lambda_2(N)\geq\cdots $$ converges to the Poisson point process on $[0,\infty)$ with intensity 
    $
     \frac{C\beta}{ x^{\beta+1}},\quad x>0,
    $ where $\lambda_1(N)\geq \lambda_2(N),\cdots\geq $ denote the eigenvalues of $G_N^\alpha$ in decreasing order.
        \end{enumerate}
\end{theorem}

\begin{remark}
    The matrix model $G_N^\infty$ \eqref{1.2001} is closely related to the Gaussian beta ensemble. The small deviations of top eigenvalue of Gaussian beta ensembles were derived in \cite{10.1214/EJP.v15-798}. Note that the precise large deviation rate function was not derived in \cite{10.1214/EJP.v15-798}. In this paper we derive the precise large deviation rate function, but for a modified model with orthogonal invariance. From a different perspective, there has been a number of recent works concerning deviations of the top eigenvalue of Gaussian beta ensembles at the high temperature limit, including \cite{ WOS:000530548400015} and \cite{pakzad2018poisson}. These works mainly consider the double scaling regime of the Gaussian beta ensemble in which $\beta\to 0$ and $N\beta\to\infty$, and is very different from the models we investigate in this paper where there is only one scaling.
\end{remark}

\subsection{Plan of the paper}\label{section2}
In Section \ref{section2} we prove Theorem \ref{1stest1st} and \ref{theorem1.1ladd}. Then in Section \ref{section3} we prove Theorem \ref{2ndtheorem}.

\section{Proof of Theorem \ref{1sttheorem}}

To prove Theorem \ref{1sttheorem}, we first need the following computation on the top eigenvalue of diagonally perturbed matrices.

\begin{lemma}\label{lemma2.1} Consider a triangular array of real numbers $(\lambda_1^{(N)},\cdots,\lambda_N^{(N)})_{N\geq 1}$ satisfying the following assumptions: there exists some (small) $c\in(0,1)$ such that 
\begin{itemize}
    \item $M_N:=\max\{\lambda_1^{(N)},\cdots,\lambda_N^{(N)}\}\in(0,\infty)$, there exists a unique $i\in [N]$ such that $M_N=\lambda_i^{(N)}$, and there exists some $\epsilon_N>0$ such that for any other $j\neq i$, $\lambda_j^{(N)}\leq M_N-\epsilon_N$. We require that $\epsilon_N$ is bounded away from zero for $N$ large, and that $M_N$ is bounded from infinity for $N$ large.

    \item At any $i\in [1,N]$ such that $|\lambda_i^{(N)}|>N^{-c}$, for any $j\in [1,N]$ where $|i-j|<N^{2c}$ the value $\lambda_j^{(N)}$ must satisfy $|\lambda_j^{(N)}|\leq N^{-c}.$
    \item There are at most $N^{1.5c}$ indices $i\in[1,N]$ such that $|\lambda_i^{(N)}|>N^{-c}.$
\end{itemize}

Then
\begin{enumerate}

\item If $M_N$ is achieved in the middle of $[1,N]$ in the sense that for any $i\in [1,\lfloor N^{c}]\rfloor\cup [N-\lfloor N^{c}\rfloor,N]$, we have $|\lambda_i^{(N)}|\leq N^{-c}$, then
$$ \lambda_1(H_N^\infty+\operatorname{diag}(\lambda_1^{(N)},\cdots,\lambda_N^{(N)}))=
\sqrt{M_N^2+4}(1+o(1))
$$
as $N$ tends to infinity, where $\lambda_1(\cdot)$ denotes the largest eigenvalue of a matrix.
\item In any case, the largest eigenvalue of $H_N^\infty+\operatorname{diag}(\lambda_1^{(N)},\cdots,\lambda_N^{(N)})$ is upper bounded by $\sqrt{M_N^2+4}(1+o(1))$ as $N$ tends to infinity.
\item Let $M_N^1:=M_N,$ $M_N^2,M_N^3,\cdots$ denote the first largest, second largest, third largest,... elements from the array $(\lambda_1^{(N)},\cdots,\lambda_N^{(N)})$. Fix any $d\in\mathbb{N}_+$, assume $M_N^1,\cdots,M_N^d$ are non-negative, uniformly bounded from infinity in $N$, and that $M_N^i-M_N^{i+1}\geq\epsilon_N$ for $i=1,\cdots,d-1$ with $\epsilon_N$ bounded away from zero. Assume moreover that for all $i\in [1,N^c]\cup [N-N^c,N]$ we have $|\lambda_i^{(N)}|\leq N^{-c}$. Then the top $d$ eigenvalues of $H_N^\infty+\operatorname{diag}(\lambda_1^{(N)},\cdots,\lambda_N^{(N)})$ are given by $\sqrt{(M_N^i)^2+4}(1+o(1))$ for $i=1,\cdots,d$.
\end{enumerate}
\end{lemma}

\begin{proof}

Denote by $\Lambda_N=\operatorname{diag}(\lambda_1^{(N)},\cdots,\lambda_N^{(N)})$ 
and $$\Lambda_N^*=\operatorname{diag}(\lambda_1^{(N)}1_{|\lambda_1^{(N)}|>N^{-c}},\lambda_2^{(N)}1_{|\lambda_2^{(N)}|>N^{-c}},\cdots,\lambda_N^{(N)}1_{|\lambda_N^{(N)}|>N^{-c}}).$$
Then as $N\to\infty$, by Cauchy interlacing theorem, $$|\lambda_1(H_N^\infty+\Lambda_N) -\lambda_1(H_N^\infty+\Lambda_N^*)|\leq N^{-c},$$ so it suffices to compute the top eigenvalue of the latter matrix. Recall that
$\lambda>2$ is an eigenvalue of $H_N^\infty+\Lambda_N^*$ if and only if 
$$\det(\lambda \operatorname{I}_N-H_N^\infty-\Lambda_N^*)=0.$$
Since we assume $\lambda>2$ and that all the eigenvalues of $H_N^\infty$ lie in $(-2,2)$ by a standard calculation, this implies $$\det(\lambda \operatorname{I}_N-H_N^\infty)\neq 0,$$ so that we have
$$
\det(\operatorname{I}_N-(\lambda \operatorname{I}_N-H_N^\infty)^{-1}\Lambda_N^*)=0.
$$
At this point we quote the computations in \cite{hu1996analytical} to compute $R_{ij}:=((\lambda \operatorname{I}_N-H_N^\infty)^{-1})_{ij}$. Denote by $\lambda^*=\operatorname{arccosh}(\frac{1}{2}\lambda)$, then the main result of \cite{hu1996analytical} states that
$$
R_{ij}=\frac{\cosh((N+1-|i-j|)\lambda^*)-\cosh((N+1-i-j)\lambda^*)}{2\sinh(\lambda^*)\sinh((N+1)\lambda^*)}.
$$
In the following we first prove the claims in item (1). We claim the following asymptotics: for any $\lambda_0>2$ we can find a constant $C(\lambda_0,N)>0$ such that, for any $i\in [N^{2c},N-N^{2c}]$, and for any $\lambda>\lambda_0$,
\begin{equation}\label{2.1first12345}|R_{ii}-\frac{1}{2\sinh(\operatorname{arccosh}(\frac{1}{2}\lambda))}|\leq C(\lambda_0,N) \to 0,\quad N\to\infty.\end{equation}
Meanwhile, for any such $i\in [N^{2c},N-N^{2c}]$, we have
\begin{equation}\label{decays1}
|R_{ij}|\leq e^{-N^{2c}}\quad \text{ for any }|j-i|>N^{2c}.
\end{equation} These two estimates can be verified from the fact that for $\lambda^*>1$, $|\cosh((N+1-i-j)\lambda^*)|\sim \frac{1}{2}e^{|N+1-i-j|\lambda_*}$ and $|\sinh((N+1-i-j)\lambda^*)|\sim \frac{1}{2}e^{|N+1-i-j|\lambda_*}$. The same asymptotics hold if we take $N+1$ and $N+1-|i-j|$ in place of $N+1-i-j$. Finally, the assumption $i\in[N^{2c},N-N^{2c}]$ ensures that $|N+1-|i-j||<N+1-N^{2c}$ and $|N+1-i-j|<N+1-N^{2c}$, and these convergence can be quantified with explicit error rates for any $\lambda>\lambda_0>2$. Combining all these asymptotics lead to \eqref{decays1}.
We can similarly verify \eqref{2.1first12345}: in this case we have $\cosh((N+1-i-j)\lambda^*)/\sinh((N+1)\lambda^*)\to 0$ as $N\to\infty$ (as we assume $i\in [N^{2c},N-N^{2c}]$, this term is actually bounded by $e^{-N^c}$) and moreover that $|\tanh((N+1)\lambda^*)- 1|\leq e^{-N^c}$ as $N\to\infty$. Combining all these asymptotics lead to \eqref{2.1first12345}.

Then for any $j\in[1,N]$ such that $|\lambda_j^{(N)}|<N^{-c}$, the $j$-th column of $\operatorname{I}_N-R\Lambda_N^*$ must have entry 1 in the $j$-th coordinate and have entry $0$ in all other coordinates. We then use Gauss elimination to simplify the matrix $\operatorname{I}_N-R\Lambda_N^*$, in the following steps: (i) for any $j_*\in[1,N]$ such that $|\lambda_{j_*}^{(N)}|>N^{-c}$, we can use Gauss elimination so that any $(i,j_*)$-th entry of $\operatorname{I}_N-R\Lambda_N^*$ with $0<|i-j_*|<N^{2c}$ is subtracted to zero, without changing the diagonal entry. This follows from subtracting a suitable multiple of the $i$-th column onto the $j$-th column to set the $(i,j_*)$ entry zero. By our assumption on $\lambda_j^{(N)}$ in the second bullet point in the statement of the lemma, this can be applied for all $|i-j_*|<N^{2c}$. (ii) The same step can be applied to all $(i,j_*)$ entry such that $|\lambda_i^{(N)}|<N^{-c}$, without changing the diagonal entry and determinant of $\operatorname{I}_N-R\Lambda_N^*$. (iii) After these procedures, we are left with a $n\times n$ matrix $T$ with the same determinant and diagonal entries as $\operatorname{I}_N-R\Lambda_N^*$, and with at most $N^{1.5c}\times N^{1.5c}$ nonzero off-diagonal entries $(i,j)$ such that $|T_{ij}|\leq e^{-N^{2c}}$.

We can now easily compute $\det(\operatorname{I}_N-R\Lambda_N^*)$. By the aforementioned Gaussian elimination procedure, this boils down to computing the determinant of a diagonal matrix (with bounded entries) perturbed by a matrix with nonzero entries only on $N^{1.5c}$ rows and columns, and that each entry of the perturbation is bounded by $e^{-N^{2c}}$.  The contribution to the determinant from the perturbation is at most $\sum_{k=1}^N (N^{1.5c})^ke^{-kN^{2c}}=O(e^{-N^c})$, as when we expand the determinant, at each step we have $N^{1.5c}$ possible choices of indices from these perturbed entries, and each perturbed entry is bounded by $e^{-N^{2c}}$.

Therefore we only need to compute the determinant of the diagonal part, and we get 

\begin{equation}\label{detdsdatqqge}
    \det(\operatorname{I}_N-R\Lambda_N^*)=\prod_{i=1}^N(1-\lambda_i^{(N)}R_{ii}1_{|\lambda_i^{(N)}|\geq N^{-c}})+O(e^{-N^{c}}).
\end{equation}

Now we can check that if $\lambda>2$ is such that $M_N(\frac{1}{2\sinh(\operatorname{arccosh}(\frac{1}{2}\lambda))}-C(\lambda_0,N))<1-\omega,$ for any fixed $\omega>0$, then $|R_{ii}\lambda_i^{(N)}|\leq 1-\omega$. This means $\det(\operatorname{I}_N-R\Lambda_N^*)>0$ for such $\lambda$. Since $\omega>0$ is arbitrary and $C(\lambda_0,N)$ tends to $0$, we conclude that for any $\omega>0$, any $\lambda>\sqrt{M_N^2+4}+\omega$ and any $N$ sufficiently large, we must have $\det(\operatorname{I_N}-R\Lambda_N^*)>0$.

Fix any $\omega_0>0$. If $\lambda>2$ is such that $M_N(\frac{1}{2\sinh(\operatorname{arccosh}(\frac{1}{2}\lambda))}-C(\lambda_0,N))>1+\omega_0/4$, but meanwhile, $(M_N-\epsilon_N)(\frac{1}{2\sinh(\operatorname{arccosh}(\frac{1}{2}\lambda))}+C(\lambda_0,N))<1-\omega_0/4$, then for any such $\lambda$ we must have $\det(\operatorname{I}_N-R\Lambda_N^*)<0$ by the assumption in the first bullet point of this lemma. The existence of such $\lambda$ is guaranteed by the fact that $M_N$ is bounded, $\epsilon_N$ is non-vanishing and $C(\lambda_0,N)$ tends to zero as $N$ gets large. Combining both facts, we conclude that the largest $\lambda$ that solves $\det(\lambda I_N-H_N^\infty-\Lambda_N^*)=0$ is $\sqrt{M_N^2+4}(1+o(1))$, so that the largest eigenvalue of $H_N^\infty+\Lambda_N^*$ must be $\sqrt{M_N^2+4}(1+o(1))$ and this concludes the proof of part (1).

For case (2), all the reductions in the previous step work as well, and the only difference is that for $i\in [1,N^c]$ or $[N-N^c,N]$ necessarily we have, as $N\to\infty$,
$$
|R_{ii}|\leq \frac{1}{2\sinh(\operatorname{arccosh}(\frac{1}{2}\lambda))}+C(\lambda_0,N),
$$ where $C(\lambda_0,N)>0$ vanishes to $0$. Adapting the proof in case (1), one sees that we must have  $\lambda_1(H_N^\infty+\Lambda_N^*)\leq\sqrt{M_N^2+4}(1+o(1))$. 

For case (3) where we study the top $d$ eigenvalues, we follow the same steps as in case (1). We again need to compute the determinant $\det(\operatorname{I_N}-R\Lambda_N^*)$ as in \eqref{detdsdatqqge}. Since by assumption $M_N^{i-1}-M_N^i\geq\epsilon_N>0$ is bounded away from zero, we can check as in \eqref{detdsdatqqge} that there is precisely one root to the equation $\det(\operatorname{I_N}-R\Lambda_N^*)$  at $\sqrt{(M_N^i)^2+4}(1+o(1))$ for each $i=1,\cdots,d$, and there are no other roots existing in the interval $(\sqrt{(M_N^d)^2+4}(1-o(1),\infty))$. This completes the proof.
\end{proof}

Now we prove Theorem \ref{1sttheorem}.

We first verify the upper tail claims in case (1), (2) and (3).

\begin{proof}[\proofname\ of Theorem \ref{1sttheorem}, upper tail of case (1)]

Assume that case (1) of Theorem \ref{1sttheorem} holds. Then for any $\lambda>0$,

\begin{equation}\label{1stest1st}\begin{aligned}
\mathbb{P}\left(\max\{N^{-\alpha} \mathfrak{a}(1),\cdots,N^{-\alpha}\mathfrak{a}(N)\}\geq  \lambda\right)&=1-\mathbb{P}(\mathfrak{a}(1)\leq N^{\alpha }\lambda)^N,\\&
=1-(1-\mathbb{P}(\mathfrak{a}(1)>N^\alpha \lambda))^N.
\end{aligned}
\end{equation}
 Noting that $\mathfrak{a}$ has the Weibull distribution, we conclude that 
\begin{equation}
    \frac{\mathbb{P}\left(\max\{N^{-\alpha} \mathfrak{a}(1),\cdots,N^{-\alpha}\mathfrak{a}(N)\}\geq  \lambda\right)}{N\exp(-CN^{\alpha\beta}\lambda^\beta)}\in [C_1,C_2](1+o(1)).
\end{equation}

Then we argue that the maximum in \eqref{1stest1st} is achieved at only one index $i$. Indeed, for any $\lambda>0$,
\begin{equation}
    \mathbb{P}(\text{ there are two indices } i: N^{-\alpha} \mathfrak{a}(i)\geq\lambda)\leq \frac{N^2-N}{2}\exp(-CN^{2\alpha\beta}\lambda^{2\beta})
\end{equation}
which is negligible compared to $\exp(-CN^{\alpha\beta})$. 

We further check the other assumptions of Lemma \ref{lemma2.1} hold with very high probability: for example, the possibility that there are indices $i,j$ such that $|i-j|<N^{2c}$ such that $|N^{-\alpha} \mathfrak{a}(i)|>N^{-c},|N^{-\alpha} \mathfrak{a}(j)|>N^{-c}$ is at most
\begin{equation}
    N\cdot N^{2c}\cdot \exp(-CN^{2(\alpha-c)\beta}),
\end{equation}
which is negligible compared to $\exp(-CN^{\alpha\beta})$ if $c>0$ is chosen small enough.  Meanwhile, the possibility that there are more than $N^{1.5c}$ sites $i$ with $|N^{-\alpha}\mathfrak{a}(i)|>N^{-c}$ is also $o(\exp(-CN^{\alpha\beta}))$ as can be seen from taking a union bound. Also, the possibility that there are two distinct $i$ and $j$ such that $|N^{-\alpha}\mathfrak{a}(i)|>\lambda$, $|N^{-\alpha}\mathfrak{a}(j)|>\lambda$ is negligible compared to the possibility that there is one $i$ with $|N^{-\alpha}\mathfrak{a}(i)|>\lambda$. This justifies that in the following computation we may assume $\epsilon_N$ is bounded away from zero (see the first bullet point in the assumption of Lemma \ref{lemma2.1}).

Now we can apply Lemma \ref{lemma2.1} to complete the proof as follows. For any $\omega>0$, lemma \ref{lemma2.1} implies that when $N$ is sufficiently large, if $\max_i N^{-\alpha}\mathfrak{a}(i)\geq\sqrt{(t+\omega)^2-4}$ and all the other conditions in Lemma \ref{lemma2.1} (1) are satisfied, then we have $\lambda_1(N)>t$.

Our discussions in the previous paragraphs imply that, conditioned on the event that $\max_i N^{-\alpha}\mathfrak{a}(i)>\sqrt{(t+\omega)^2-4}$, then all the other assumptions in the three bullet points of Lemma \ref{lemma2.1} are satisfied with probability $1-o(1)$ and can be made uniform for all $t>t_0>0$. Therefore we compute: for $N$ sufficiently large,

\begin{equation}\begin{aligned}\label{heavier1}
    \mathbb{P}(\lambda_1(N)>t)&\geq \mathbb{P}(\max_i N^{-\alpha}\mathfrak{a}(i)>{\sqrt{(t+\omega)^2-4}})\\&\cdot\mathbb{P} (\text{Assumptions of Lemma \ref{lemma2.1} (1) satisfied}\mid \max_i N^{-\alpha}\mathfrak{a}(i)>\sqrt{(t+\omega)^2-4})
\\&\geq N\exp(-CN^{\alpha\beta}(\sqrt{(t+\omega)^2-4})^\beta)(1+o(1))
    \end{aligned}
\end{equation} 
We also utilized the fact that $$\mathbb{P}(\operatorname{argmax}_i N^{-\alpha}\mathfrak{a}(i)\in [N^c,N-N^c])=1+o(1).$$ This completes the proof of the lower bound.

Now we prove the reverse inequality. Fix $t>2$ and $\omega>0$ sufficiently small. The idea is that the possibility that the assumptions in the three bullet points of Lemma \ref{lemma2.1} do not hold is negligible compared to the probability that $\mathbb{P}(\max_i N^{-\alpha}\mathfrak{a}(i)>{\sqrt{(t-\omega)^2-4}})$, by our previous discussion. Then, assuming these assumptions are justified, Lemma \ref{lemma2.1} (2) implies that if none of the $N^{-\alpha}\mathfrak(a)(i)\geq\sqrt{(t-\omega)^2-4}$, then we must have $\lambda_1(N)<t$. Written more formally, we have deduced that 
for any $\omega>0$, when $N$ is sufficiently large, 
\begin{equation}\label{heavier2}\begin{aligned}
    \mathbb{P}(\lambda_1(N)>t)&\leq \mathbb{P}(\max_i N^{-\alpha}\mathfrak{a}(i)>{\sqrt{(t-\omega)^2-4}})\\&+\mathbb{P}(\text{Assumptions in Lemma \ref{lemma2.1} not all satisfied with }\lambda_i=N^{-\alpha}\mathfrak{a}(i))
    \\&=N\exp(-CN^{\alpha\beta}(\sqrt{(t-\omega)^2-4})^\beta)(1+o(1)).\end{aligned}
\end{equation} This completes the proof of case (1), the upper tail.\end{proof}

\begin{proof}[\proofname\ of Theorem \ref{1sttheorem}, upper tail of case (2)]
Now we consider case (2), where the potentials $\mathfrak{a}(i)$ have heavier tails. We see that 

\begin{equation} \label{case2stars}
    \begin{aligned}
\mathbb{P}\left(\max\{N^{-\alpha} \mathfrak{a}(1),\cdots,N^{-\alpha}\mathfrak{a}(N)\}\geq  \lambda\right)
&=1-(1-\mathbb{P}(\mathfrak{a}(1)>N^\alpha \lambda))^N
\\&=1-(1-CN^{-\alpha\beta}\lambda^{-\beta})^N\\&=CN^{1-\alpha\beta}\lambda^{-\beta} (1+o(1)),
\end{aligned} 
\end{equation} recalling the assumption that $\alpha\beta>1$.
Similar to the previous case, one can check that conditioning on the event that \eqref{case2stars} holds, the assumptions of Lemma \ref{lemma2.1} hold with probability $1+o(1)$. Therefore we have the analogue of estimates \eqref{heavier1}, \eqref{heavier2} in this setting of heavier tails: for any $\omega>0$, when $N$ is sufficiently large, 
\begin{equation}
    \mathbb{P}(\lambda_1(N)>t)\geq CN^{1-\alpha\beta}\sqrt{(t+\omega)^2-4}^{-\beta} (1+o(1)),
\end{equation}
and for any $\omega>0$, when $N$ is sufficiently large, 
\begin{equation}\begin{aligned}
    \mathbb{P}(\lambda_1(N)>t)&\leq CN^{1-\alpha\beta}\sqrt{(t-\omega)^2-4}^{-\beta} (1+o(1))
    \\&+\mathbb{P}(\text{Assumptions in Lemma \ref{lemma2.1} not all satisfied with }\lambda_i=N^{-\alpha}\mathfrak{a}(i))
    \\&\leq CN^{1-\alpha\beta}\sqrt{(t-\omega)^2-4}^{-\beta} (1+o(1)),
\end{aligned}\end{equation} justifying the upper tail in case (2). 
\end{proof}

\begin{proof}[\proofname\ of Theorem \ref{1sttheorem}, upper tail of case (3)]
Now we consider case (3). Under the critical moment condition \eqref{1.171},
\begin{equation} \label{case3stars}
    \begin{aligned}
\mathbb{P}\left(\max\{N^{-\alpha} \mathfrak{a}(1),\cdots,N^{-\alpha}\mathfrak{a}(N)\}\geq  \lambda\right)
&=1-(1-\mathbb{P}(\mathfrak{a}(1)>N^\alpha \lambda))^N
\\&=1-(1-CN^{-1}\lambda^{-\frac{1}{\alpha}})^N\\&\to 1-e^{-C\lambda^{-\frac{1}{\alpha}}},\quad N\to\infty. 
\end{aligned} 
\end{equation}
Moreover, when the claimed event \eqref{case3stars} occurs, it is easy to check that the assumptions in Lemma \ref{lemma2.1} hold with probability $1-o(1)$. Therefore we have the following two-sided inequalities: for any $t>2$ and any $\omega>0$ sufficiently small, then for $N$ sufficiently large,
\begin{equation}\begin{aligned}
  \left(1-e^{-c(\sqrt{(t+\omega)^2-4})^{-\frac{1}{\alpha}}}\right)(1+o(1))\leq    \mathbb{P}(\lambda_1(N)>t)\leq \left( 1-e^{-c(\sqrt{(t-
  \omega)^2-4})^{-\frac{1}{\alpha}}}\right)(1+o(1)).
\end{aligned}\end{equation}Setting $\omega>0$ to be arbitrarily small, the upper tail claim in case (3) now follows.
\end{proof}

\begin{proof}[\proofname\ of Theorem \ref{1sttheorem}, Poisson distribution in case (4)] The argument is similar to \cite{Soshnikov2004}. 

We first compute that for any $\lambda>0$,
\begin{equation} \label{case3starssfdfs}
    \begin{aligned}
\mathbb{P}\left(\max\{N^{-\frac{1}{\beta}} \mathfrak{a}(1),\cdots,N^{-\frac{1}{\beta}}\mathfrak{a}(N)\}\geq  \lambda\right)
&=1-(1-\mathbb{P}(\mathfrak{a}(1)>N^\alpha \lambda))^N
\\&=1-(1-CN^{-1}\lambda^{-\frac{1}{\beta}})^N\\&\to 1-e^{-C\lambda^{-\frac{1}{\beta}}},\quad N\to\infty. 
\end{aligned} 
\end{equation}
That is, 
$$
\mathbb{P}(\max_i N^{-\alpha}\mathfrak{a}(i)\geq N^{\frac{1}{\beta}-\alpha}\lambda)\to 1-e^{-C\lambda^{-\frac{1}{\beta}}},
$$
and one can easily check that the second largest value among $N^{-\frac{1}{\beta}}\mathfrak{a}(i)$ is strictly smaller than the first largest, with probability $1-o(1)$. Assuming that $i\in [n]$ is the site where the maximum is achieved in \eqref{case3starssfdfs}, then we can choose a unit vector such that $f_i=(0,\cdots,0,1,\cdots,0)$, i.e. $f_i$ is 1 on the $i$-th coordinate and is zero everywhere else. Then one easily sees, assuming $\max_i N^{-\frac{1}{\beta}}\mathfrak{a}(i)=\lambda$, that
$$
H_N^\alpha f_i=N^{\frac{1}{\beta}-\alpha}\lambda f_i(1+o(1)),
$$
so by perturbation theorem of Hermitian operators, $N^{\alpha-\frac{1}{\beta}}H_N^\alpha$ has en eigenvalue at $\lambda(1+o(1))$. Meanwhile, observe that for such a choice of $\lambda$, we have $N^\frac{1}{\beta}-\lambda$ is the value of the matrix $H_N^\sigma$ (which is also sparse), so the largest eigenvalue of $N^{\alpha-\frac{1}{\beta}}H_N^\alpha$ cannot exceed $\lambda(1+o(1))$. This completes the proof of part (4).

\end{proof}

\begin{proof}[\proofname\ of Theorem \ref{1sttheorem}, Lower tails in cases (1),(2) and (3).] 

The last step we need to verify is that the possibility for $\lambda_1(N)$ to be smaller than 2 is much smaller than the possibility specified in each case (1),(2),(3), justifying our claims on the left tail. 

For this purpose, note that if $\lambda_1(N)<2-\delta<2$ for some $\delta>0$, then $$d_2(\mu_{H_N^{\alpha}},\mu_{H_N^\infty})>C(\delta)>0$$ for some constant $C(\delta)>0$, where $d_2$ denotes the 2-Wasserstein distance on $\mathbb{R}$ defined for two probability measures $\mu,\nu$ via
$$
d_2(\mu,\nu)=\inf_{\pi\in\mathcal{C}(\mu,\nu)}\left(\int_{\mathbb{R}^2} |x-y|^2d\mu(x,y)\right)^{1/2},
$$(with $\mathcal{C}(\mu,\nu)$ the space of probability measures on $\mathbb{R}^2$ with first marginal $\mu$ and second marginal $\nu$), 
and $\mu_{H_N^{\alpha}}$,$\mu_{H_N^\infty}$ respectively denote the empirical measure of eigenvalues of $H_N^\alpha$ and $H_N^\infty$.

Using the trivial coupling, we have
\begin{equation}\label{trivialcoupling}\begin{aligned}
d_2(\mu_{H_N^{\alpha}},\mu_{H_N^\infty})&\leq (\frac{1}{N}\sum_{i=1}^N |\lambda_i^\alpha-\lambda_i^\infty|^2)^{1/2}
\\&\leq \frac{1}{\sqrt{N}}(\operatorname{Tr}(H_N^\infty-H_N^\alpha)^2)^{\frac{1}{2}}
\\&=\frac{1}{N^{\alpha+\frac{1}{2}}}\left(\sum_{i=1}^N |\mathfrak{a}(i)|^2\right)^{\frac{1}{2}}.
\end{aligned}\end{equation}
where the last step follows from Hoffman-Wielandt inequality, see for example \cite{anderson2010introduction}, Lemma 2.1.19 and also \cite{dolai2021ids}. 

We begin with case (1) of Theorem \ref{1stest1st}, where $\mathfrak{a}(\cdot)$ is assumed to have a Weibull distribution.
First assume (1a) that $\beta=2$. Then we proceed with the following computation: by Markov's inequality, for any $d>0$,
\begin{equation}\begin{aligned}\label{markovkeys}
\mathbb{P}(\sum_{i=1}^N |\mathfrak{a}(i)|^2\geq N^{2\alpha+1})&\leq e^{-CdN^{2\alpha+1}}\mathbb{E}\left[\exp(d\sum_{i=1}^N |\mathfrak{a}(i)|^2)\right]
\\&=  e^{-CdN^{2\alpha+1}}\mathbb{E}[\exp\left(d|\mathfrak{a}(i)|^2\right)]^{N}\\&\leq C_2e^{-C_3 N^{2\alpha+1}}
\end{aligned}\end{equation}
for some constants $C_2,C_3>0$, where the last step we used the tail decay of the law of $\mathfrak{a}$.

Then assume (1b) that $\beta>2$, then from Jensen's inequality 
\begin{equation}
    \sum_{i=1}^N |\mathfrak{a}(i)|^2 \leq \left(\sum_{i=1}^N |\mathfrak{a}(i)|^\beta\right)^\frac{2}{\beta} N^{1-\frac{2}{\beta}}
\end{equation}
we similarly deduce that, for a different $C'>0$,
\begin{equation}\begin{aligned}
    \mathbb{P}(\sum_{i=1}^N |\mathfrak{a}(i)|^2\geq CN^{2\alpha+1})\leq \mathbb{P}(\sum_{i=1}^N |\mathfrak{a}(i)|^\beta\geq C'N^{\frac{\beta}{2}(2\alpha+\frac{2}{\beta})}  )
\leq C_2 e^{-C_3N^{\alpha\beta+1}},
\end{aligned}
\end{equation} where in the last steep we used Markov's inequality (see \eqref{markovkeys}) and the tail estimate in the law of $\mathfrak{a}$. 

Finally assume (1c) that $\beta<2$. In this case the above reasoning is no longer effective. We end up with a much weaker estimate of the following form: for some $c>0$ to be determined,
\begin{equation}\begin{aligned}
    \mathbb{P}(\sum_{i=1}^N |\mathfrak{a}(i)|^2\geq CN^{2\alpha+1})&\leq \mathbb{P}(\max_{i\in[1,N]}|\mathfrak{a}(i)|\geq CN^{\alpha+c})\\&+\mathbb{P}(\sum_{i=1}^N|\mathfrak{a}(i)|^\beta\geq CN^{2\alpha+1-(2-\beta)(\alpha+c)})\\&\leq C_2e^{-C_3N^{(\alpha+c)\beta}}+C_2e^{-C_3N^{\alpha\beta+1+(\beta-2) c}},
    \end{aligned} 
\end{equation} and a careful choice of (sufficiently small) $c>0$ completes the proof.

Now we consider case (2) of Theorem \ref{1stest1st}, i.e. the random variables $\mathfrak{a}(\cdot)$ have heavy tails yet $\alpha\beta>1$. Fix some sufficiently small $\epsilon>0$, then $\mathfrak{a}$ has finite $\beta-\epsilon$-th moment. Denote by $\beta'=\beta-\epsilon$. In case (2a), (2b) that $\beta'\geq 2$, we get similar to the previous proof that 
\begin{equation}
    \mathbb{P}(\sum_{i=1}^N |\mathfrak{a}(i)|^2\geq CN^{2\alpha+1})\leq \mathbb{P}(\sum_{i=1}^N |\mathfrak{a}(i)|^{\beta'}\geq C'N^{\frac{\beta'}{2}(2\alpha+\frac{2}{\beta'})})\leq (C_2N)^{-(\alpha\beta'+1)}.
\end{equation}
In case (2c) that $\beta'<2$, we deduce as previous that for sufficiently small $c>0$,
\begin{equation}\begin{aligned}
    \mathbb{P}(\sum_{i=1}^N |\mathfrak{a}(i)|^2\geq CN^{2\alpha+1})&\leq \mathbb{P}(\max_{i\in[1,N]}|\mathfrak{a}(i)|\geq CN^{\alpha+c})\\&+\mathbb{P}(\sum_{i=1}^N|\mathfrak{a}(i)|^{\beta'}\geq CN^{2\alpha+1-(2-{\beta'})(\alpha+c)})\\&\leq C_2N^{1-\beta'(\alpha+c)}+(C_2N)^{-(\alpha\beta+1+(\beta-2) c)}.
    \end{aligned} 
\end{equation} 
This completes the proof.

\begin{remark}\label{explicitlefttail}
    We can also obtain the explicit dependence on $\lambda$ concerning the left tail rate function. Notice that $\mu_{H_N^\infty}$ converges to the arcsine law $\mu_{as}$ on $[-2,2]$, we can deduce that if $\lambda_1(N)\in(-2,2),$ then 
$$d_2^2(\mu_{H_N^{\alpha}},\mu_{as})\geq \int_{\lambda_1(N)}^2 \frac{(x-\lambda_1(N))^2dx}{\pi\sqrt{4-x^2}}$$
and when $\lambda_1(N)<-2$ we similarly have 
$$
d_2(\mu_{H_N^{\alpha}},\mu_{as})\geq -2-\lambda_1(N).
$$
Then we can plug this more precise estimate into \eqref{trivialcoupling}, \eqref{markovkeys} and apply Markov's inequality. Overall we get a left tail estimate with explicit dependence on $\mu$.

The same procedure applies to the matrix model $G_N^\alpha$, and one only needs to replace the arcsine law $\mu_{as}$ by the semicircle law $\mu_{sc}$, which is the limiting spectral measure of $G_N^\infty.$
    
\end{remark}

\end{proof}
Finally we prove Theorem \ref{theorem1.1ladd}. The ideas are similar to the proof of \cite{Soshnikov2004}, Theorem 1.2. We begin with a lemma:

\begin{lemma}  For any $0<c_1<d_1<c_2<d_2<\cdots<c_k<d_k<\infty$, let $I_l=(c_l,d_l)$.

Then in case (3) of Theorem \ref{1stest1st}, the number of $N^{-\alpha}\mathfrak{a}(i)$, $i\in[N]$ in $(I_l)_{l=1}^k$ forms a Poisson point process on $[0,\infty)$ with intensity $\frac{C}{\alpha x^{\frac{1}{\alpha}+1}}$ on $[0,\infty)$.

In case (4) of Theorem \ref{1stest1st}, the number of $N^{-\frac{1}{\beta}}\mathfrak{a}(i)$ in $(I_l)_{l=1}^k$ forms a Poisson point process on $[0,\infty)$ with intensity $\frac{C\beta}{ x^{\beta+1}}$ on $[0,\infty)$.
    
\end{lemma}
This lemma is close to Proposition 1 of \cite{Soshnikov2004}. Its verification is a generalization of computations in \eqref{case3stars} and \eqref{case3starssfdfs} and can be found for example in \cite{leadbetter2012extremes}, Theorem 2.3.1.

\begin{proof}[\proofname\ of Theorem \ref{theorem1.1ladd}] In case (3) of Theorem \ref{theorem1.1ladd}, from an elementary computation of random variables, we see that for any $\epsilon>0$, and $\omega'>0$, with probability $1-\omega'$ there are only a (bounded in $N$, depending on $\omega'$) number of $i\in[N]$ such that $N^{-\alpha}\mathfrak{a}(i)>\epsilon$, and with probability $1-\omega'$ the configuration $(\lambda_i^{(N)}=N^{-\alpha}\mathfrak{a}(i))_{i\in[N]}$ satisfies the assumptions of Lemma \ref{lemma2.1} (3), and that any $i$ such that $N^{-\alpha}\mathfrak{a}(i)>\epsilon$ must satisfy $i\in [N^c,N-N^c]$ for a small $c$. In this case, the computations in Lemma \ref{lemma2.1}(3) provide us with a one-to-one correspondence between the large $N$-limit of some $N^{-\alpha}\mathfrak{a}(i)>\epsilon$ and the large $N$-limit of a large eigenvalue of $H_N^\alpha$ via the map $x\mapsto\sqrt{x^2+4}$. This completes the proof.

    In case (4) of Theorem \ref{theorem1.1ladd}, the idea of proof is similar to \cite{Soshnikov2004}, Theorem 1.2. We claim that the large $N$-limit of the configuration of $N^{-\frac{1}{\beta}}\mathfrak{a}(i)$ precisely correspond to the large $N$-limit of the largest eigenvalues of $N^{\alpha-\frac{1}{\beta}}H_N^\alpha$. This correspondence has been verified for the top eigenvalue in the proof of Theorem \ref{1sttheorem}, and repeating that proof for the second, third, etc. largest point in the limit (considering approximate eigenfunctions supported on only one index, and using perturbation to show existence of an eigenvalue) shows that each these limit points correspond to the $N\to\infty$ limit of an eigenvalue of $H_N^\alpha$, so we have identified $k$ distinct values as the $N\to\infty$ limit of eigenvalues of $N^{\alpha-\frac{1}{\beta}}H_N^\alpha$. To show that they are indeed the \textit{largest} $k$ limiting eigenvalues of $N^{\alpha-\frac{1}{\beta}}H_N^\alpha$, consider the $k$ indices $i$ such that $N^{-\frac{1}{\beta}}\mathfrak{a}(i)$ achieves the $k$ largest values for $i\in[1,N]$ and consider $H_N^{\alpha,k}$, which is $H_N^\alpha$ with these $k$ rows and columns removed. Then the largest eigenvalue of $N^{\alpha-\frac{1}{\beta}}H_N^{\alpha,k}$ converges to the $k+1$-st largest point in the limiting configuration of $N^{-\frac{1}{\beta}}\mathfrak{a}(i)$. Meanwhile, by interlacing, this limit is at least as large as the $k+1$-th largest eigenvalue of $N^{\alpha-\frac{1}{\beta
    }}H_N^\alpha$. Thus we have set up the correspondence between the top $k$ eigenvalues of $N^{\alpha-\frac{1}{\beta
    }}H_N^\alpha$ and the largest $k$ points in the Poisson point process, for each $k$.
\end{proof}

\section{Proof of Theorem \ref{2ndtheorem}}\label{section3}

The proof of Theorem \ref{2ndtheorem} is essentially the same as that of Theorem \ref{1sttheorem}, except that we need to prove the following result on unitarily perturbed matrices.

\begin{lemma}\label{lemma3.1woc} Recall the deterministic matrix $G_N^\infty$ defined in \eqref{1.2001}.
    Let $$\Lambda_N=\operatorname{diag}(\lambda_1^{(N)},\cdots,\lambda_N^{(N)}),$$ and let $U_N$ be a Haar distributed unitary matrix. Let $M_N:=\max_{i\in [N]}\lambda_i^{(N)}>0$. Assume that for any $\epsilon>0$ there are only finitely many $i\in[1,N]$ (the number does not grow with $N$) such that $|\lambda_i^{(N)}|>\epsilon$. Assume that 
    $$
\lim_{N\to\infty}M_N=M>0.
    $$
    Then as $N\to\infty$, the largest eigenvalue of $$G_N^\infty+U_N\Lambda_N U_N^T$$ converges almost surely to 
    \begin{equation}\label{3.11.33.11.3}
        \begin{cases}
2, \quad M<1,\\
M+\frac{1}{M},\quad M\geq 1.
        \end{cases}
    \end{equation}
\end{lemma}

\begin{proof}
    Consider the matrix $G_N^\beta$ defined as

\begin{equation}\label{fucky}
    G_N^\beta=\frac{1}{\beta}
    \begin{pmatrix}
    N(0,2)&\chi_{(N-1)\beta}&0&\cdots&\cdots&0\\
    \chi_{(N-1)\beta}&N(0,2)& \chi_{(N-2)\beta}&0&\cdots&0\\
   \vdots &\ddots&\ddots&\ddots&\ddots&\ddots\\
    0&\cdots&\chi_{3\beta}&N(0,2)&\chi_{2\beta}&0\\
   0&\cdots &0&\chi_{2\beta}&N(0,2)&\chi_\beta\\
    0&\cdots&\cdots&0&\chi_\beta&N(0,2)
    \end{pmatrix},
\end{equation}
where the $N(\cdot)$ are centered normal distributions and the $\chi_\cdot$ are chi distributions with the specified parameter. Thanks to subGaussian tails of these random variables, once we have proved that the top eigenvalue of $G_N^\beta+U_N\Lambda_N U_N^T$ converges almost surely to some deterministic limit as $N\to\infty$, the same conclusion holds for the top eigenvalue of $G_N^\infty+U_N\Lambda_N U_N^T$. 

The matrix $G_N^\beta$ is the matrix representation for Gaussian beta ensembles obtained in Dumitriu-Edelman \cite{dumitriu2002matrix}. In particular, the empirical measure of eigenvalues of $G_N^\beta$ converges to the semicircle law on $[-2,2]$ as $N\to\infty$ and that the largest and smallest eigenvalues of $G_N^\beta$ converge to $-2$ and $2$ respectively. 

Thanks to our assumption on the magnitude of $\lambda_i$, we may assume without loss of generality that $\Lambda_N$ has finite rank (that is, we set $\lambda_i$ to be 0 if $|\lambda_i|<c$, and the resulting top eigenvalue of the modified matrix differs by at most $c$ to the top eigenvalue of the original matrix via Cauchy interlacing formula, and we finally set $c\to 0$).

Since $U_N\Lambda_N U_N^T$ is orthogonal invariant, we may use the main result of \cite{benaych2011eigenvalues} to conclude that the top eigenvalue of $G_N^\beta+U_N \Lambda_N U_N^T$ converges to \eqref{3.11.33.11.3} almost surely. Although the paper \cite{benaych2011eigenvalues} requires the spike matrix has eigenvalues independent of $N$, which is not assumed here, the result in \cite{benaych2011eigenvalues} can still be applied here as we can take $N$ sufficiently large. We only need the following modifications: as the largest eigenvalue of the spiked matrix only depends on $M_N^1$ but not the rest, we perturb $M_N^1$ to be $M+\epsilon$ (such that $M+\epsilon>M_N^1$, and later set $\epsilon\to 0$,) and apply \cite{benaych2011eigenvalues}. The effect of such replacement to the top eigenvalue converges to $0$ thanks to Cauchy interlacing theorem. Then by the previous argument, the top eigenvalue of $G_N^\infty+U_N\Lambda_N U_N^T$ converges to the same limit.
    
\end{proof}

Now we complete the proof of Theorem \ref{2ndtheorem}.

\begin{proof}
    The proof of Theorem \ref{2ndtheorem} is exactly the same as the proof of Theorem \ref{1sttheorem} and \ref{theorem1.1ladd}, and one simply needs to replace Lemma \ref{lemma2.1} by Lemma \ref{lemma3.1woc}. The proof of upper tails is exactly the same via Hoffman-Wielandt inequality, so we concentrate on proving the lower tails.
    
    For the lower tails, we follow exactly the same computation as in Theorem \ref{1sttheorem}. For the upper tails in cases (3), we only need to check that the assumptions of Lemma \ref{lemma3.1woc} are satisfied with probability $1-o(1)$: indeed,
we may work under assumption \eqref{1.171} where $\mathfrak{a}$ has the heaviest tails and deduce that
\begin{equation}
    \mathbb{P}(\text{ there exists $k$ different } i\in[N]: N^{-\frac{1}{\alpha}}|\mathfrak{a}(i)|>\epsilon)\leq C_\epsilon \tbinom{k}{N} (N)^{-k}\leq \frac{C_\epsilon}{k!}\to 0,\quad k\to\infty,
\end{equation}
so that for any $\epsilon>0$, we may assume thee are a (bounded in $N$) number of subscripts $i$ such that $N^{-\frac{1}{\alpha}}|\mathfrak{a}(i)|>\epsilon$. Then by Lemma \ref{lemma3.1woc}, the largest eigenvalue is governed by the largest value of the potentials $N^{-\frac{1}{\alpha}}\mathfrak{a}(i),i\in [N].$ The distribution of this maximum is already derived in \eqref{1stest1st}, \eqref{case2stars} and \eqref{case3stars}. This completes the proof.

For the upper tail in case (2), we check that for any $\epsilon>0$,
\begin{equation}
    \mathbb{P}(\text{ there exists two different } i\in[N]: N^{-\frac{1}{\alpha}}|\mathfrak{a}(i)|>\epsilon)=O(N^2N^{-2\alpha\beta})=o(N^{-\alpha\beta+1}) 
\end{equation}  and hence negligible in the limit \eqref{agagagagagagrqag}.

For the upper tail in case (1), we check that for any $\epsilon>0$,
\begin{equation}
    \mathbb{P}(\text{ there exists two different } i\in[N]: N^{-\frac{1}{\alpha}}|\mathfrak{a}(i)|>\epsilon)=O(N^2e^{-2CN^{\alpha\beta}}) 
\end{equation}  and hence negligible in the limit \eqref{limitsss}.
    
\end{proof}

\section*{Acknowledgement}
The author has no financial or non-financial conflicting of interests to declare that are relevant to the contents of this article.

\printbibliography
\end{document}